\documentclass[11pt,a4paper]{article}

\usepackage[T1]{fontenc}
\usepackage[utf8]{inputenc}
\usepackage[english]{babel}
\usepackage{amsmath,amssymb,amsthm,mathtools}
\usepackage[a4paper,margin=25mm]{geometry}
\usepackage{microtype}
\usepackage[hidelinks]{hyperref}
\usepackage{enumitem}
\usepackage{xcolor}

\renewcommand\le{\leqslant}
\renewcommand\ge{\geqslant}
\setlist{nosep}

\newtheorem{theorem}{Theorem}[section]

\newtheorem{exttheorem}{Theorem}

\newtheorem{proposition}[theorem]{Proposition}
\newtheorem{lemma}[theorem]{Lemma}
\newtheorem{corollary}[theorem]{Corollary}
\theoremstyle{remark}
\newtheorem{remark}[theorem]{Remark}

\newcommand{\C}{\mathbb C}
\newcommand{\R}{\mathbb R}
\newcommand{\Z}{\mathbb Z}
\newcommand{\T}{\mathbb T}
\newcommand{\E}{\mathbb E}
\newcommand{\PP}{\mathbb P}
\newcommand{\avg}{\mathrm{avg}}
\newcommand{\eps}{\varepsilon}
\newcommand{\Id}{\mathrm{Id}}
\newcommand{\1}{\mathbf 1}

\DeclareMathOperator{\rank}{rank}
\DeclareMathOperator{\diag}{diag}
\DeclareMathOperator{\dist}{dist}
\DeclareMathOperator{\sign}{sign}
\DeclareMathOperator{\Span}{span}

\title{Low-dimensional approximation of uniformly bounded orthonormal systems}
\author{Yuri Malykhin}
\date{}

\begin{document}
\maketitle

\begin{abstract}
We prove that if $f_1,\ldots,f_N$ is an uniformly bounded orthonormal system, then
for all $p\in[1,2)$ there is an estimate for its Kolmogorov widths:
$d_n(\{f_1,\dots,f_N\},L_p) \gtrsim \min\{1,n^{-1/p}N^{1/2}\}$, $n\le N/2$.
In the regime $n\asymp N$ this provides the lower bound $N^{-\alpha_p}$ with
a sharp exponent $\alpha_p:=1/p-1/2$. Besides that, it follows that a good
approximation of such ONS requires the dimension $n\gtrsim N^{p/2}$.

Our second result is an approximation theorem for Fourier
matrices of finite abelian groups (this includes the usual DFT matrices). For every $\eta>0$ there is
$a=a(\eta)>0$ such that any Fourier matrix $F$ of sufficiently large order has
an approximation of rank $N^{1-a}$ with the row-wise $\ell_1$-error
at most $N^{1/2+\eta}$; the exponent $1/2$ is sharp.  Consequences include the 
approximation of the same kind for circulant matrices; the
approximation of the trigonometric functions $\exp(2\pi i\langle
\lambda,x\rangle)$, $\lambda\in \Lambda=K\cap\mathbb{Z}^d$, $K$ is symmetric convex,
    with dimension $|\Lambda|^{1-a}$ and an optimal error
$|\Lambda|^{-\alpha_p+\eta}$; bounds for widths of weighted Wiener classes.
\end{abstract}

\section{Introduction and main results}

Let $f_1,\ldots,f_N$ be an orthonormal system of functions. We study the
$L_p$-approximation of these functions by linear spaces of some dimension $n$
that is essentially smaller than $N$. In terms of Kolmogorov widths the
approximation error equals
\begin{equation}
\label{widths_problem}
d_n(\{f_1,\ldots,f_N\}, L_p)
:= \inf_{\dim V\le n}\;\max_{1\le k\le N}\dist_{L_p}(f_k,V).
\end{equation}
This paper continues the line of
work~\cite{Malykhin2022,Malykhin2024,MalRyutin,AstMal};
now the research is empowered by AI.

The considered problem may seem rather specific~---
usually one is interested in approximating a wide class of functions using
elements of given ONS, not the approximation of the ONS elements themselves.
However, as we will see, this problem is indeed substantive, contains interesting examples and
has some applications. It is connected to classical notions in
Harmonic analysis, low-rank matrix approximation and matrix rigidity, and
widths of finite-dimensional bodies and functional compacts.

The starting point for the problem~\eqref{widths_problem} was the well-known
equality in the Euclidean space:
\begin{equation}
    \label{l2_rigidity}
    d_n(\{f_1,\ldots,f_N\},L_2) = \sqrt{1-n/N},
\end{equation}
which implies that the width is bounded away from zero if $n\le N(1-\eps)$; hence
the system is ``rigid'' (poorly approximated).
In~\cite{Malykhin2022} it was proven that the first elements of the Walsh
system are not rigid in $L_p$, $p<2$.
Various sufficient conditions for the rigidity were obtained
in~\cite{Malykhin2024}: for $1<p<2$ any $p'$-lacunary ONS is rigid in $L_p$, and
any ONS of independent functions is rigid in $L_1$.  The rigidity of independent
systems was further studied for
rearrangement-invariant spaces in~\cite{AstMal}. In
papers~\cite{MalRyutin,MalRyutin2} the lower bounds for
widths of finite systems were applied to widths of finite-dimensional sets.
Now let us proceed to the new results.
We will consider only uniformly bounded systems that are rigid
in $L_2$ and, of course, in $L_p$ for $p>2$. So we restrict to the
case $1\le p<2$.

By $d_n$ we denote the classical Kolmogorov width in real normed spaces; by
$d_n^{\mathbb{C}}$ we will denote its complex version.

\begin{theorem}[Lower bound for uniformly bounded O.N.S]
\label{thm_intro_lower}
Let $(\Omega,\mu)$ be a probability space, let $1\le p<2$, and let
$f_1,\dots,f_N$ be a real orthonormal system in $L_2(\Omega)$ such that
$\|f_k\|_\infty\le K$, $k=1,\dots,N$.
Then, for every $1\le n\le N/2$,
\begin{equation}\label{eq:intro-quantitative-lower}
 d_n(\{f_1,\dots,f_N\},L_p)
    \ge c(p,K)\min\left\{1,n^{-1/p}N^{1/2}\right\}.
\end{equation}
For complex orthonormal systems the same inequality holds for the width
    $d_n^{\mathbb{C}}$, $n\le N/4$.
\end{theorem}

In particular,
$$
d_{N/2}(\{f_1,\ldots,f_N\}, L_p)\gtrsim N^{-\alpha_p},
\quad\alpha_p:=\frac1p-\frac12.
$$
The exponent $\alpha_p$ is sharp, as we will see soon (we will keep this
notation in the rest of the paper).
The second corollary is that the approximation of ONS with a fixed sufficiently small error
requires the dimension $n\gtrsim N^{p/2}$; we do not know if the exponent $p/2$ is optimal here.

There is also a new sufficient condition for the rigidity in $L_1$.
Let $G$ be a compact abelian group with normalized Haar measure and let
$\Lambda=\{\gamma_1,\dots,\gamma_N\}\subset\widehat G$ be a finite set of
distinct characters.  We say that $\Lambda$ is a \emph{Sidon set with
constant $S$} if
\begin{equation}\label{eq:intro-sidon}
 \sum_{j=1}^N|a_j|
 \le S\left\|\sum_{j=1}^Na_j\gamma_j\right\|_{C(G)}
\end{equation}
for all complex coefficients $a_j$.

\begin{theorem}[Rigidity of Sidon characters]\label{thm:sidon-rigidity}
Suppose that $\gamma_1,\dots,\gamma_N$ belong to a Sidon set with constant
$S$.  Then for every $\delta\in(0,1)$ and every
$n\le(1-\delta)N$,
\[
 d_n^{\mathbb C}(\{\gamma_1,\dots,\gamma_N\},L_1(G))
 \ge \frac{c(\delta)}{S}.
\]
\end{theorem}

\begin{corollary}[Hadamard-lacunary systems]\label{cor:lacunary-sidon}
For every $q>1$ and $\delta\in(0,1)$ there is $c(q,\delta)>0$ such that, if
$1\le k_1<k_2<\cdots$ and $k_{j+1}/k_j\ge q$, then for every
$n\le(1-\delta)N$,
\[
 d_n^{\mathbb C}
 \bigl(\{e^{2\pi i k_jx}\}_{j=1}^N,L_1(\T)\bigr)
 \ge c(q,\delta).
\]
\end{corollary}

It is still unknown how far can we go from independence keeping $L_1$-rigidity.
Is the Rademacher chaos rigid in $L_1$?

\paragraph{Fourier matrices.}
We now turn to upper bounds. For an $N\times N$ matrix $A$ we consider the
Frobenius norm $\|A\|_F :=
(\sum_{i,j}|A_{i,j}|^2)^{1/2}$ and
the row-wise norms
$$
\|A\|_{p,\infty} := \max_{1\le i\le N} \left(\sum_{j=1}^N |A_{i,j}|^p\right)^{1/p}.
$$
The row-wise low-rank approximation is exactly a Kolmogorov-width problem.
See~\cite{KMR} for more examples in the max-norm.
It is different from (but related to) classical entrywise matrix rigidity, originating in
\cite{Valiant1977}; in particular, the non-rigidity theorems of Alman and
Williams~\cite{AlmanWilliams} (for Walsh--Hadamard matrices) and of Dvir and Liu
\cite{DvirLiu2020} (for Fourier and circulant matrices) concern the classical
rigidity metric.

The well--known Eckart--Young theorem states that
$$
\min_{\rank B\le n}\|A-B\|_F = (\sum_{k>n}\sigma_k^2(A))^{1/2}.
$$
Therefore for any unitary matrix $U$ and any matrix $B$ with rank $n$ we have
\begin{equation}
    \label{unitary}
\|U-B\|_{2,\infty} \ge N^{-1/2}\|U-B\|_F \ge \sqrt{1-n/N}.
\end{equation}
This is, of course, the same lower bound as in~\eqref{l2_rigidity}.

%For $v\in\C^N$ put
%\[
% \norm{v}_{L_p^N}
% :=\left(\frac1N\sum_{j=1}^N|v_j|^p\right)^{1/p},
%\]
%and for an $N\times N$ matrix $A$ set
%\[
% \Delta_p(A):=
% \max_{1\le i\le N}
% \left(\frac1N\sum_{j=1}^N|A_{ij}|^p\right)^{1/p},
% \qquad
% \Delta_\infty(A):=\max_{i,j}|A_{ij}|.
%\]

Let $G$ be a finite abelian group.  Fix a non-degenerate symmetric pairing
\[
 G\times G\longrightarrow\R/\Z,
 \qquad (x,y)\longmapsto x\cdot y,
\]
and use it to identify $G$ with its dual group.  The corresponding Fourier
matrix is
\[
    \mathcal{F}_G:=\bigl(e^{2\pi i x\cdot y}\bigr)_{x,y\in G}.
\]
Changing the self-duality only permutes rows and columns.  For
the cyclic groups $G=\Z_N$ this is the usual DFT matrix (note a slight abuse
of the notation here)
\[
    \mathcal{F}_N:=\mathcal{F}_{\Z_N}=\bigl(e^{2\pi ixy/N}\bigr)_{x,y=0}^{N-1}.
\]

Another important example is the dyadic group $G=\mathbb{Z}_2^d$. The
corresponding Fourier matrix 
is the Walsh--Hadamard matrix with entries
$(-1)^{\langle x,y\rangle}$, $x,y\in\mathbb{Z}_2^d$.

\begin{theorem}[Approximation of Fourier matrices]
\label{thm:main-rank}
Let $\eta>0$.  There are constants
$a=a(\eta)>0$, $N_0=N_0(\eta)$,
such that for every finite abelian group $G$ of order $N\ge N_0$ there is a
matrix $B_G$ of order $N$ satisfying
$$
 \rank B_G\le N^{1-a}
    \quad\mbox{and}\quad
    \|\mathcal{F}_G-B_G\|_{1,\infty}\le N^{1/2+\eta}.
$$
\end{theorem}

The matrix $B_G$ provides a nontrivial approximation for every fixed $p<2$:
$\|\mathcal{F}_G-B_G\|_{p,\infty} \le \|\mathcal{F}_G-B_G\|_{1,\infty} \le N^{1/2+\eta}$. Note
that the exponent $1/2$ is sharp. Indeed,
since $N^{-1/2}\mathcal{F}_G$ is unitary, we have for any $n$-rank matrix $B$:
$$
\|\mathcal{F}_G-B\|_{2,\infty} \ge (N-n)^{1/2}.
$$

Let us state Theorem~\ref{thm:main-rank} in terms of widths. We identify rows of
$\mathcal{F}_G$ with characters $\chi_x\colon y\mapsto e^{2\pi ix\cdot y}$, $\chi_x\in
\widehat{G}$. Therefore,
$$
\min_{\rank B\le n}\|\mathcal{F}_G-B\|_{p,\infty} = N^{1/p}d_n^{\mathbb{C}}(\widehat{G},L_p(G)).
$$
We obtain the following result with a sharp exponent.
\begin{corollary}
    Let $1\le p<2$, $\eta>0$. For any finite abelian group $G$ of order $N\ge
    N_0(\eta)$ we have
$$
d^{\mathbb C}_{N^{1-a}}(\widehat{G},L_p(G)) \le N^{-\alpha_p+\eta},
    \quad a=a(\eta).
$$
\end{corollary}

\paragraph{Consequences.}
We record several consequences of Theorem~\ref{thm:main-rank} which motivate the
approximation problem from different directions.

First, the one-dimensional statement extends to convex frequency sets in
fixed dimension.  For a finite set $\Lambda\subset\Z^d$ write
\[
 \mathcal E_\Lambda
 :=\{e_\lambda:\lambda\in\Lambda\},
 \qquad
 e_\lambda(x):=e^{2\pi i\langle\lambda,x\rangle},
 \quad x\in\T^d.
\]

\begin{proposition}[Centrally symmetric convex spectra]
\label{prop:intro-convex-spectrum}
Let $1\le p<2$, let $\eta>0$, and let $d\ge1$ be fixed.  There are
constants $a=a(\eta)>0$ and $N_0=N_0(d,\eta)$ such that the following
holds.  If $K\subset\R^d$ is an origin-symmetric convex body,
\[
 \Lambda:=K\cap\Z^d,
 \qquad
 N:=|\Lambda|\ge N_0,
\]
then
\[
 d_{N^{1-a}}^{\mathbb C}
 \bigl(\mathcal E_\Lambda,L_p(\T^d)\bigr)
 \le N^{-\alpha_p+\eta}.
\]
\end{proposition}

A qualitative form of non-rigidity of trigonometric functions was first proven by
A.~Sinai as an answer to the question posed in~\cite{Malykhin2024}
(personal communication). Namely, he proved that
$$
d_n(\{e^{ikx}\}_{k=-N}^N,L_1)=o(1)\quad\mbox{for $n=o(N)$.}
$$
See also the extended
version~\cite{Sinai} where the $L_p$ case is also considered (our work was done independent of this paper).

A second consequence concerns structured matrices. Again, let $G$ be a finite abelian
group. A $G$-circulant is a matrix
$$
\mathcal{C}_{G,f}(x,y) = f(x-y),\quad x,y\in G,
$$
defined by a function $f\colon G\to\C$. The usual circulant matrices correspond
to $G=\Z_N$.

\begin{proposition}[Bounded $G$-circulants]
\label{prop:intro-circulant}
Let $\eta>0$.  There are constants
$a=a(\eta)>0$ and $N_0=N_0(\eta)$ such that, whenever
$|G|=N\ge N_0$ and $|f|\le1$, there is a matrix $B$ satisfying
$\rank B\le N^{1-a}$ and
$$
    \|\mathcal{C}_{G,f}-B\|_{1,\infty}\le N^{1/2+\eta}.
$$
\end{proposition}

This statement is motivated by a work of Dvir and Liu~\cite{DvirLiu2020}.

The circulant result also gives low-rank bilinear approximation of 
kernels generated by trigonometric polynomials.

\begin{corollary}[Bilinear approximation]
\label{cor:intro-translation-kernels}
Let $\eta>0$.  There are constants $a=a(\eta)>0$ and $N_0=N_0(\eta)$
such that the following holds.  For every trigonometric polynomial
\[
 T(t)=\sum_{|k|\le N}c_k e^{2\pi i kt},\qquad N\ge N_0,
\]
there exist an integer $r\le N^{1-a}$ and trigonometric polynomials
$u_1,\ldots,u_r,v_1,\ldots,v_r$ such that, simultaneously for all
$1\le p<2$,
\[
 \left\|T(x-y)-\sum_{\nu=1}^r u_\nu(x)v_\nu(y)\right\|_{L_p(\T^2)}
 \le C_\eta\|T\|_\infty N^{-\alpha_p+\eta}.
\]
\end{corollary}

Compare Temlyakov's Problem~5.6 in~\cite{Temlyakov2003}.

Finally, these finite-system estimates have consequences for weighted
coefficient classes. The periodic weighted Wiener ball is
\[
 \mathcal{A}_\beta
 :=\left\{f\in L_1(\T)\colon \sum_{k\in\Z}(1+|k|)^\beta|\widehat{f}(k)|\le1\right\}.
\]
Weighted Wiener spaces are classical smoothness spaces defined by absolute
summability of weighted Fourier coefficients.  Nguyen, Nguyen and Sickel
\cite{NguyenNguyenSickel2022} studied Kolmogorov and other $s$-numbers of
embeddings of weighted Wiener algebras, with particularly complete results
for the $L_2$ target space.
More recently, Moeller, Stasyuk and
Ullrich \cite{MoellerStasyukUllrich2026} considered best $m$-term
trigonometric approximation and sampling recovery in weighted Wiener
spaces.

% There is one more article on the subject,
% Y.~Chen, X.~Pan, Y.~Xu and G.~Chen,
%``The Approximation Characteristics of Weighted $p$-Wiener Algebra''
% The estimates in L_q, q<2, go through L_2 so nothing new for us

\begin{corollary}[Weighted Wiener class]
\label{cor:intro-trig-wiener}
    For every $1\le p<2$ and $\beta>0$ there are $B=B(p,\beta)>0$ and $b=b(p,\beta)>0$ such that
\[
    n^{-\beta-B}
    \lesssim d_n^{\mathbb{C}}(\mathcal{A}_\beta,L_p(\T))
 \lesssim n^{-\beta-b}.
\]
    Moreover, one may take $B=\alpha_p\min\{1,2\beta\}$, $b=c(p)\min\{1,\beta\}$.
\end{corollary}

The rest of the paper is organized as follows.  Section~2 proves the
quantitative lower bound and the Sidon rigidity theorem.  Section~3 proves
Theorem~\ref{thm:main-rank}.
Section~4 develops the consequences stated above.

\section{Lower bounds and rigidity}

\subsection{A quantitative lower bound}

We need the following definition from~\cite{Hazan}.
Let $M$ be a compact in $\R^d$.
A finite set of vectors $v_1,\ldots,v_r\in M$ is called a \emph{volumetric spanner} for $M$ if
every vector $v\in M$ has a representation
\begin{equation}\label{eq:volumetric-spanner}
v = \sum_{k=1}^r c_k v_k,
\quad \mbox{with } \sum_{k=1}^r c_k^2 \le 1.
\end{equation}

Hazan and Karnin~\cite{Hazan} prove that every compact set has a volumetric spanner of size
$r\le 12d$. The notion of volumetric spanner is related to well--known
convex geometry concepts such as Auerbach basis and L\"owner--John ellipsoid.

Although any linear bound for the size of the volumetric spanner is sufficient for us, we also cite a recent
paper~\cite{BhaskaraMahabadiVakilian2023} by Bhaskara, Mahabadi and Vakilian, with a more precise result and a
number of relevant references. The following statement is contained
in~\cite[Theorem~3.6]{BhaskaraMahabadiVakilian2023}.

\begin{exttheorem}
    \label{thm_volumetric}
    Any finite set $M\subset\R^d$ has a volumetric
    spanner of size at most $2d-1$.
\end{exttheorem}

Now we are ready to prove Theorem~\ref{thm_intro_lower}.

\begin{proof}
Fix a subspace $V$ with $\dim V=n$ and choose approximants $g_k\in V$ such that
\[
 \|f_k-g_k\|_p\le\varepsilon,
 \qquad k=1,\dots,N.
\]
Denote $\tau:=n\eps^p/N^{p/2}$. We are going to prove that either $\eps\ge
c(p,K)$ or $\tau\ge c(p,K)$;
in the latter case we have $\eps \gtrsim n^{-1/p}N^{1/2}$.

Identify the $n$-dimensional space $V$ linearly with $\R^n$.
    Consider the set $\{g_k\}_{k=1}^N\subset V$ and take a volumetric spanner
    for it: $g_{j_1},\dots,g_{j_r}$, $r\le 2n-1$, using Theorem~\ref{thm_volumetric}.
    By the representation~\eqref{eq:volumetric-spanner} we obtain
    for all $k$ that
$$
 |g_k(x)|
 \le\left(\sum_{s=1}^r|g_{j_s}(x)|^2\right)^{1/2} 
 \le\left(\sum_{s=1}^r|f_{j_s}(x)|^2\right)^{1/2}
    +\left(\sum_{s=1}^r|f_{j_s}(x)-g_{j_s}(x)|^2\right)^{1/2} 
 \le K\sqrt r+R(x),
$$
where
\[
    R(x):=\left(\sum_{s=1}^r|f_{j_s}(x)-g_{j_s}(x)|^2\right)^{1/2}.
\]
Since $p\le2$,
\[
    \|R\|_p^p = \int_\Omega\left(\sum_{s=1}^r|f_{j_s}(x)-g_{j_s}(x)|^2\right)^{p/2}\,d\mu
    \le\sum_{s=1}^r\|f_{j_s}-g_{j_s}\|_p^p
    \le2n\varepsilon^p.
\]
Consider the set
\[
    E:=\{x\colon R(x)\le K\sqrt N\}.
\]
Then
\[
    \mu(E^c)\le \frac{2n\eps^p}{(K\sqrt N)^p} = 2K^{-p}\tau.
\]
Since $r\le 2n\le N$, on $E$ we have
\begin{equation}
    \label{uniform_fg_over_E}
 |g_k|\le2K\sqrt N,
 \qquad
 |f_k-g_k|\le3K\sqrt N.
 \end{equation}

Put
\[
 I:=\sum_{k=1}^N\int_E f_kg_k\,d\mu.
\]
We give a lower bound on $I$ using the restriction $\|f_k\|_\infty\le K$ and
    the approximation bound
    $\|f_k-g_k\|_1 \le \|f_k-g_k\|_p\le \eps$:
$$
I = \sum_{k=1}^N \int_\Omega f_k^2\,d\mu - \sum_{k=1}^N\int_{E^c}f_k^2\,d\mu
- \sum_{k=1}^N \int_E f_k(f_k-g_k)\,d\mu
    \ge N - N K^2 \cdot 2K^{-p}\tau - NK\eps.
$$
Therefore,
\begin{equation}\label{eq:lower-correlation}
 \frac{I}{N}
 \ge1-K\varepsilon-2K^{2-p}\tau.
\end{equation}

Now we obtain an upper bound on $I$. Let
\[
 W:=\operatorname{span}\{g_1\mathbf1_E,\dots,g_N\mathbf1_E\}
 \subset L_2(\Omega).
\]
Then $\dim W\le n$.  If $P_W$ is the orthogonal projection onto $W$,
Bessel's inequality gives
\[
 \sum_{k=1}^N\|P_Wf_k\|_2^2\le n.
\]
Therefore
\[
 I
 = \sum_{k=1}^N \int_\Omega P_W f_k \cdot g_k\1_E\,d\mu
 \le \sum_{k=1}^N \|P_W f_k\|_2\cdot\|g_k\1_E\|_2
 \le\sqrt n\left(\sum_{k=1}^N\|g_k\mathbf1_E\|_2^2\right)^{1/2}.
\]
By the triangle inequality,
$$
 \left(\sum_{k=1}^N\|g_k\mathbf1_E\|_2^2\right)^{1/2}
 \le \left(\sum_{k=1}^N\|f_k\1_E\|_2^2\right)^{1/2}
    + \left(\sum_{k=1}^N\|(f_k-g_k)\1_E\|_2^2\right)^{1/2}.
$$
It is clear that the first term is at most $N^{1/2}$. We bound the second term
using that $\|f_k-g_k\|_p\le\eps$ and
the uniform bound~\eqref{uniform_fg_over_E} on $f_k-g_k$ over $E$:
$$
\sum_{k=1}^N \|(f_k-g_k)\1_E\|_2^2 \le N\cdot \eps^p (3K\sqrt{N})^{2-p}.
$$
Summing up, we obtain
\begin{equation}\label{eq:lower-bessel}
 \frac{I}{N}
    \le\sqrt{\frac nN}+\sqrt{n}(3K)^{1-p/2}\eps^{p/2}N^{-p/4} = \sqrt{\frac
    nN}+(3K)^{1-p/2}\sqrt\tau.
\end{equation}
Combining \eqref{eq:lower-correlation} and \eqref{eq:lower-bessel},
\[
 1-\sqrt{\frac nN}
 \le K\varepsilon+2K^{2-p}\tau+(3K)^{1-p/2}\sqrt\tau.
\]
Since $n\le N/2$, the left-hand side is bounded from below.
Thus, with a constant depending only on $p$ and $K$, either
$\varepsilon\ge c$ or $\tau\ge c$.

    The required bound~\eqref{eq:intro-quantitative-lower} 
    follows since $V$ and $g_k$ are arbitrary.

    The complex case follows by a realification argument.
Consider the probability space
$\widetilde\Omega=\Omega\times\{1,2\}$
and put
$\widetilde{f}_k(x,1):=\sqrt2\,\operatorname{Re}f_k(x)$,
$\widetilde{f}_k(x,2):=\sqrt2\,\operatorname{Im}f_k(x)$.
Then $(\widetilde f_k)$ is a real uniformly bounded orthonormal system and
\begin{equation}\label{eq:intro-realification}
 d_n^{\mathbb C}(\{f_k\},L_p)
 \ge d_{2n}(\{\widetilde f_k\},L_p).
\end{equation}
\end{proof}

\subsection{Sidon character systems}

In the proof of Theorem~\ref{thm:sidon-rigidity} it will be convenient to use
the average Kolmogorov width
\[
 d_n^{\avg,\mathbb C}(\{f_1,\dots,f_N\},L_1)
 :=\inf_{\dim_{\mathbb C}V\le n}
   \frac1N\sum_{j=1}^N\dist_{L_1}(f_j,V).
\]
For every finite complex family $f_1,\dots,f_N\in L_1(\Omega,\mu)$,
\begin{equation}\label{eq:sidon-transposition}
 N d_n^{\avg,\mathbb C}(\{f_1,\dots,f_N\},L_1)
 =\inf_{\substack{V\subset\C^N\\ \dim_{\mathbb C}V\le n}}
   \int_\Omega
   \dist_{\ell_1^N}\bigl((f_1(\omega),\dots,f_N(\omega)),V\bigr)
   \,d\mu(\omega).
\end{equation}
Indeed, the coordinates of any measurable $V$-valued map span a subspace
of $L_1$ of dimension at most $\dim V$, and conversely the vector formed by
$N$ approximants from a common $n$-dimensional subspace takes values in an
$n$-dimensional subspace of $\C^N$.
See~\cite{Malykhin2024} for more details.

We will require one short lemma. Recall that Steinhaus variable is a complex random
variable with uniform distribution on the unit circle $|z|=1$.
\begin{lemma}
\label{lem_steinhaus}
If $z=(z_1,\dots,z_N)$ is a vector of independent Steinhaus variables  and
$V\subset\C^N$ satisfies $\dim_{\mathbb C}V\le(1-\delta)N$, then
\begin{equation}\label{eq:steinhaus-rigidity}
 \E \dist_{\ell_1^N}(z,V)\ge c(\delta) N.
\end{equation}
\end{lemma}

\begin{proof}
Suppose that $z$ is approximated by a random vector $w(\omega)\in V$.
Put
$s_k'=\sign(\operatorname{Re}z_k)$ and
$s_k''=\sign(\operatorname{Im}z_k)$.  The $2N$ variables
$s_1',s_1'',\ldots,s_N',s_N''$ are independent Rademacher variables; denote
$s=(s_1',s_1'',\ldots,s_N',s_N'')$.  We will use that
$\E(z_k\mid s)=\E(z_k\mid s_k',s_k'')=(2/\pi)(s_k'+is_k'')$.

Consider the conditional expectation
$$
    \E(\|z-w\|_1\;|\;s)
    = \sum_{k=1}^N \E(|z_k-w_k|\;|\;s)
    \ge \sum_{k=1}^N |\frac{2}{\pi}(s_k'+is_k'') - \E(w_k|s)|
    \gtrsim \sum_{k=1}^N(|s_k' - w_k'|+|s_k''-w_k''|),
$$
    where $w_k':=(\pi/2)\operatorname{Re}\E(w_k|s)$ and
    $w_k'':=(\pi/2)\operatorname{Im}\E(w_k|s)$.
    Note that the random vector $(w_1',w_1'',\ldots,w_N',w_N'')$ lies in the
    $2n$-dimensional real subspace
    $$
    \widetilde{V}=\{(\operatorname{Re}v_1,\operatorname{Im}v_1,\ldots,\operatorname{Re}v_N,\operatorname{Im}v_N)\colon
    v\in V\}.
    $$
    Therefore
\[
    \E\|z-w\|_1
    \gtrsim \E\dist_{\ell_1^{2N}}(s,\widetilde{V})
    \gtrsim_\delta N,
\]
where the last estimate is the usual rigidity of independent Rademacher
variables; see~\cite[Corollary~3.1]{Malykhin2024}.
\end{proof}

All necessary facts on complex measures and analysis on locally-compact abelian groups may be
found in classical books of W.~Rudin:~\cite[Chapter 6]{RudinRealComplex},~\cite{RudinGroups}.

We prove Theorem~\ref{thm:sidon-rigidity}.
\begin{proof}
Fix a complex subspace $V\subset\C^N$ and put
$$
\rho(u):=\dist_{\ell_1^N}(u,V).
$$
In view of~\eqref{eq:sidon-transposition}, our goal is to give a lower bound on the average distance
$$
\int_G \rho(\Gamma(x))\,dm_G(x),
\quad \mbox{where }\Gamma(x):=(\gamma_1(x),\dots,\gamma_N(x))
$$
and $m_G$ being the Haar measure (normalized by $m_G(G)=1$).

Now fix $b=(b_1,\dots,b_N)$ with $|b_j|\le1$.  On
$E=\operatorname{span}\{\gamma_1,\dots,\gamma_N\}\subset C(G)$ define
\[
    L_b\!\left(\sum_{j=1}^N a_j\gamma_j\right):=\sum_{j=1}^N a_jb_j.
\]
By \eqref{eq:intro-sidon}, $\|L_b\|\le S$.  Hahn--Banach and the Riesz
representation theorem give a complex Borel measure $\nu_b$ on $G$ such
that
\begin{equation}\label{eq:sidon-interpolation-measure}
 \int_G\gamma_j(h)\,d\nu_b(h)=b_j,
 \quad j=1,\dots,N,
\end{equation}
and the total variation $|\nu_b|(G)\le S$.
Consequently, for every $x\in G$,
\begin{equation}\label{eq:sidon-vector-interpolation}
    \int_G \Gamma(x+h)\,d\nu_b(h)
    = \Gamma(x)\circ b = (b_1\gamma_1(x),\ldots,b_N\gamma_N(x)).
\end{equation}

Complex measure $d\nu_b$ admits a polar decomposition: $d\nu_b = \phi\, d|\nu_b|$,
where $|\phi(x)|\equiv 1$.
We use~\eqref{eq:sidon-vector-interpolation} and Jensen's inequality for the
convex function $\rho$ and the measure $|\nu_b|$:
\begin{multline*}
\rho(\Gamma(x)\circ b)
= \rho\left(\int_G \Gamma(x+h)\,d\nu_b(h)\right) 
= \rho\left(\int_G \phi(h)\Gamma(x+h) \,d|\nu_b|(h)\right) \le \\
\le \int_G \rho(\phi(h)\Gamma(x+h))\,d|\nu_b|(h)
= \int_G \rho(\Gamma(x+h))\,d|\nu_b|(h).
\end{multline*}

Integrating over $x\in G$, using Haar invariance and
$|\nu_b|(G)\le S$, we obtain
\begin{equation}\label{eq:sidon-quotient-comparison}
    \int_G \rho\bigl(\Gamma(x)\circ b\bigr)\,dm_G(x)
    \le \int_G\int_G \rho(\Gamma(x+h))\,dm_G(x)\,d|\nu_b|(h)
    \le S\int_G \rho(\Gamma(x))\,dm_G(x).
\end{equation}

Average \eqref{eq:sidon-quotient-comparison} over independent uniform
phases $b=z$.  For fixed $x$, the vector $\Gamma(x)\circ z$ has the same
distribution as $z$.  Hence, whenever
    $\dim_{\mathbb C}V\le(1-\delta)N$, Lemma~\ref{lem_steinhaus} gives
\[
    \int_G \rho(\Gamma(x))\,dm_G(x)
 \ge\frac{c(\delta)}{S}N.
\]
Recall that $\rho$ is the $\ell_1^N$-distance to $V$.
Taking the infimum over $V$ and using \eqref{eq:sidon-transposition},
\[
 d_n^{\avg,\mathbb C}(\{\gamma_1,\dots,\gamma_N\},L_1(G))
 \ge\frac{c(\delta)}{S}.
\]
Since the usual Kolmogorov width dominates the average width, theorem is proved.
\end{proof}

\begin{proof}[Proof of Corollary~\ref{cor:lacunary-sidon}]
The classical theorem on Hadamard-lacunary sets says that, for every fixed
$q>1$, a sequence satisfying $k_{j+1}/k_j\ge q$ is a Sidon subset of
$\widehat\T=\Z$ with Sidon constant depending only on $q$; see
\cite{RudinGroups}.  Apply Theorem~\ref{thm:sidon-rigidity}.
\end{proof}

\section{Fourier matrices approximation}
\label{sec_fourier}

In this section we will prove Theorem~\ref{thm:main-rank}. For the proof it
would be convenient to normalize the row $\ell_p$-norms.
For an $N\times N$ matrix we denote
$$
\Delta_p(A) := N^{-1/p}\|A\|_{p,\infty} = \max_{1\le i\le N} \left(\frac1N\sum_{j=1}^N
|A_{i,j}|^p\right)^{1/p},
$$
with the usual convention that $\Delta_\infty(A)=\max_{i,j}|A_{i,j}|$.
Thus Theorem~\ref{thm:main-rank} is equivalent to constructing a matrix $B_G$
of rank at most $N^{1-a}$ such
that 
$$
\Delta_1(\mathcal{F}_G-B_G) \le N^{-1/2+\eta}.
$$

We use the following property.
If $G=G_1\times\cdots\times G_d$, then, up to row and column permutations,
\begin{equation}\label{eq:intro-tensor-Fourier}
    \mathcal{F}_G=\mathcal{F}_{G_1}\otimes\cdots\otimes \mathcal{F}_{G_d},
\end{equation}
the tensor (Kronecker) product.
Throughout this section we use its basic properties:
\begin{equation}
 \label{eq:Delta-multiplicative}
 \Delta_1(A\otimes B)=\Delta_1(A)\Delta_1(B),
 \qquad
 \rank(A\otimes B)=\rank A\,\rank B.
\end{equation}

Let us briefly introduce the ideas of the proof. First of all,
\begin{equation}
 \label{fourier_4}
    \mathcal{F}_G^4=N^2\mathrm{Id}.
\end{equation}
Therefore, $\mathcal{F}_G$ can have only 4 distinct
eigenvalues: $\pm N^{1/2}$, $\pm iN^{1/2}$. Let $\lambda$ be the eigenvalue with
maximal multiplicity. Define $B:=\mathcal{F}_G-\lambda\,\mathrm{Id}$. Then $\rank B\le
3N/4$ and $\Delta_1(\mathcal{F}_G-B)=N^{-1/2}$. To go further we have to decrease the
rank. Suppose that $N=q_1\cdots q_d$ where $q_i$ are distinct primes. Write the expansions
$\mathcal{F}_{q_i}=B_i+E_i$. Then the matrix $B_1\otimes\cdots\otimes B_d$ has rank
at most $N(3/4)^d$ and the approximation error may be controlled. Finally, we
pass to arbitrary $N$. This was the
approach taken by A.~Sinai (personal communication). Our proof follows the same
main ideas but is more involved.

We remark that the properties of tensor products with respect to approximation in
Hamming metric were thoroughly studied by specialists in Matrix Rigidity,
see~\cite{DvirLiu2020,Alman21,Kivva}.

\subsection{General tensor-product estimates}

We start we some general results applicable to arbitrary matrices.
Let $H_i$ be square matrices of
order $N_i$, $i=1,\ldots,d$, such that
\begin{equation}
    \label{common_tensor}
    \Delta_1(H_i)\le1,
 \qquad
 H_i=B_i+E_i,
 \qquad
 \rank B_i\le\rho_iN_i,
 \qquad
    \Delta_1(E_i)\le\delta_i.
\end{equation}

We repeatedly use the following expansion.
Put
\[
 H:=\bigotimes_{i=1}^dH_i,
 \qquad
 N:=\prod_{i=1}^dN_i,
\]
and expand
\begin{equation}
 \label{common_expansion}
 H=\sum_{S\subset[d]}T_S,
 \qquad
 T_S:=\bigotimes_{i\in S}B_i\otimes
      \bigotimes_{j\notin S}E_j.
\end{equation}
For a family $\mathcal S\subset2^{[d]}$, let
$B_{\mathcal S}:=\sum_{S\in\mathcal S}T_S$.  Then
\begin{equation}
 \label{common_expansion_rank}
 \rank B_{\mathcal S}
 \le N\sum_{S\in\mathcal S}\prod_{i\in S}\rho_i,
\end{equation}
and
\begin{equation}
 \label{common_expansion_error}
    \Delta_1(H-B_{\mathcal S})
 \le
 \sum_{S\notin\mathcal S}
 \prod_{i\in S}(1+\delta_i)
 \prod_{j\notin S}\delta_j.
\end{equation}
The next two lemmas correspond to two useful choices of $\mathcal S$.

\begin{lemma}[Equal-gain tensor truncation]
\label{lem:tensor-truncation}
    Suppose that the relations~\eqref{common_tensor} hold with
    $\rho_1=\ldots=\rho_d=\rho\in(0,1)$ and $\delta_1=\ldots=\delta_d=\delta\in(0,1)$.
For every integer $0\le s<d$, there is a matrix $B$
of order $N$ such that
\begin{equation}\label{eq:equal-block-rank}
 \rank B\le N 2^d\rho^{d-s}
\end{equation}
and
\begin{equation}\label{eq:equal-block-error}
    \Delta_1(H-B) \le 4^d\delta^{s+1}.
\end{equation}
\end{lemma}

\begin{proof}
    Use the expansion~\eqref{common_expansion} and let $B$ be the
    sum of the terms containing
    at most $s$ factors $E$. This corresponds to the family
    $\mathcal{S}=\{S\subset\{1,\ldots,d\}\colon |S|\ge d-s\}$.

    Let us bound the rank in~\eqref{common_expansion_rank}:
\[
 \rank B \le N\sum_{k=d-s}^d\binom dk\rho^k
 \le N 2^d\rho^{d-s}.
\]
    The error bound~\eqref{common_expansion_error} simplifies to
\[
    \Delta_1(H-B)
    \le\sum_{k=0}^{d-s-1}\binom dk(1+\delta)^k\delta^{d-k}
    \le 2^d(1+\delta)^d\delta^{s+1}.
\]
\end{proof}

\begin{lemma}[Power gain]
\label{lem:polynomial-tensor-closure}
Fix $a,b>0$, $0<\sigma<1$, and $\zeta>0$.  There is
$L=L(a,b,\sigma,\zeta)\ge2$ with the following property.  Suppose
$H_i=B_i+E_i$ are square matrices of order $N_i\ge L$ and
\[
    \Delta_1(H_i)\le1,
 \qquad
 \rank B_i\le N_i^{1-a},
 \qquad
    \Delta_1(E_i)\le N_i^{-b}.
\]
For $H:=\bigotimes_iH_i$ and $N:=\prod_iN_i$, there is a matrix $B$ such
that
\[
 \rank B\le N^{1-a\sigma/2},
 \qquad
    \Delta_1(H-B)\le N^{-b(1-\sigma)+\zeta}.
\]
\end{lemma}

\begin{proof}
    We will use that $N\ge L^d$, so $d\le\log N/\log L$ and $2^d\le N^{\log
    2/\log L}$.

For $S\subset[d]$ put $N_S:=\prod_{i\in S}N_i$ and retain in $B$ the terms $T_S$
    with $N_S\ge N^\sigma$. This corresponds to $\mathcal{S}=\{S\colon
    N_S\ge N^\sigma\}$. For every retained term,
\[
 \rank T_S\le NN_S^{-a}\le N^{1-a\sigma}.
\]
The sum of their ranks is at most
\[
 2^dN^{1-a\sigma}
 \le N^{1-a\sigma+\log 2/\log L}.
\]
Choose $L$ so large that $\log 2/\log L\le a\sigma/2$.

For any discarded term, $N_S<N^\sigma$, hence $N/N_S>N^{1-\sigma}$ and
\[
    \Delta_1(T_S)
 \le
 N^{-b(1-\sigma)}
 \prod_{i\in S}(1+N_i^{-b})
 \le N^{-b(1-\sigma)}e^{dL^{-b}}
    \le N^{-b(1-\sigma) + L^{-b}/\log L}.
\]
Summing over at most $2^d$ subsets gives
\[
    \Delta_1(H-B)
 \le
 N^{-b(1-\sigma)+(\log 2+L^{-b})/\log L}.
\]
Increase $L$ so that the last positive exponent is at most $\zeta$.
\end{proof}

\subsection{Spectral decomposition}

We use the refined version of~\eqref{fourier_4}. Let $R_G$ denote the permutation matrix of the involution $x\mapsto-x$ on
$G$. A direct computation gives
\begin{equation}
 \label{eq:F-square}
 \mathcal{F}_G^2=|G|R_G.
\end{equation}

\begin{lemma}[Spectral decomposition]
\label{lem:local-spectral}
Let $G$ be a finite abelian group of order $N$.  There is
a decomposition
\[
 \mathcal{F}_G=A_G+E_G
\]
such that
\[
 \rank A_G\le\frac N2,
 \qquad
 \Delta_1(E_G)\le(N/2)^{-1/2}.
\]
Moreover, $\mathcal{F}_G,A_G,E_G$ have a common orthonormal eigenbasis, and on each
vector $v$ of this basis exactly one of the following alternatives holds:
\begin{equation}
 \label{eq:local-dichotomy}
 A_Gv=0,\quad \mathcal{F}_Gv=E_Gv,
 \qquad\text{or}\qquad
 A_Gv=-2E_Gv,\quad \mathcal{F}_Gv=-E_Gv.
\end{equation}
\end{lemma}

\begin{proof}
Put $F:=\mathcal{F}_G$ and $R:=R_G$ from~\eqref{eq:F-square}. By the definition, $R$ and $F$ commute
and are normal. On
\[
 H_+:=\ker(R-\Id),
 \qquad
 H_-:=\ker(R+\Id),
\]
the eigenvalues of $F$ belong respectively to
$\{\pm\sqrt N\}$ and $\{\pm i\sqrt N\}$.  Choose an eigenvalue
$\lambda_+$ of maximal multiplicity on $H_+$ and an eigenvalue $\lambda_-$
of maximal multiplicity on $H_-$.  If $H_-=\{0\}$, choose either of the two
numbers $\pm i\sqrt N$.  Define
\[
 a:=\frac{\lambda_++\lambda_-}{2},
 \qquad
 b:=\frac{\lambda_+-\lambda_-}{2},
 \qquad
 E_G:=a\Id+bR,
 \qquad
 A_G:=F-E_G.
\]
Let us check that the matrix $A_G$ vanishes on the selected eigenspaces.
    Indeed, if $v\in H_+$, $Fv=\lambda_+v$, then
    $A_Gv=Fv-av-bRv=\lambda_+v-(\lambda_++\lambda_-)/2\,v-(\lambda_+-\lambda_-)/2\,v=0$;
    the other case is analogous.
So, $A_G=0$ on subspaces of total dimension at least $N/2$,
and $\rank A_G\le N/2$. 

Note that $|a|=|b|=\sqrt{N/2}$.  Every row of $E_G$ has
at most two nonzero entries; if the two positions coincide, the resulting
entry has modulus $\sqrt N$. Hence
\[
 \Delta_1(E_G) \le \frac{2(N/2)^{1/2}}{N} =(N/2)^{-1/2}.
\]

On the selected eigenspace in each $H_\pm$ one has $A_G=0$, hence $F=E_G$.
On the complementary eigenspace, the eigenvalue of $F$ is the negative of
the selected one, whereas the eigenvalue of $E_G$ is unchanged.  Thus
$F=-E_G$ and $A_G=-2E_G$ there.
\end{proof}

This lemma immediately gives the optimal approximation error $N^{-1/2}$, but with
rank $N/2$. In order to reduce the rank we will use tensor products.

\subsection{Bounded Fourier products}

We will use the following property of tensor products. If
$\{v^{(i)}_1,\ldots,v^{(i)}_{N_i}\}$, $i=1,\ldots,d$, are orthonormal eigenbases
for matrices $H_1,\ldots,H_d$, respectively, then the set of vectors
\begin{equation}
\label{tensor_eigenbasis}
v = v^{(1)}_{j_1}\otimes \cdots \otimes v^{(d)}_{j_d},
\qquad 1\le j_i\le N_i,\quad i=1,\ldots,d,
\end{equation}
is an orthonormal eigenbasis for $H = H_1\otimes\cdots\otimes H_d$ and the
eigenvalues equal the product of the corresponding eigenvalues:
$$
Hv = \lambda^{(1)}_{j_1}\cdots\lambda^{(d)}_{j_d}v,
\quad H_iv^{(i)}_{j_i} = \lambda^{(i)}_{j_i}.
$$

\begin{lemma}[Bounded products]
\label{lem:balanced-products}
Suppose
\[
 G=G_1\times\cdots\times G_d,
 \qquad
 N:=|G|,
 \qquad
 N_{\max}:=\max_{1\le i\le d}|G_i|,
\]
where $|G_i|\ge2$.  Let $1\le k\le d/2$.  Then there is a matrix $B_G$ such
that
\begin{equation}
 \label{eq:balanced-general-rank}
 \rank B_G\le N\cdot 2\exp\left(-c\frac{k^2}{d}\right)
\end{equation}
and
\begin{equation}
 \label{eq:balanced-general-error}
 \Delta_1(\mathcal{F}_G-B_G)
 \le
 N^{-1/2}\,2^d
 \left(\frac{C d^2}{k^2}\right)^k
    (N_{\max})^{k/2}.
\end{equation}
\end{lemma}

\begin{proof}
    For every $i$, take $\mathcal{F}_{G_i}=A_i+E_i$ from
Lemma~\ref{lem:local-spectral} and put
$\delta_i:=(|G_i|/2)^{-1/2}$.
    Thus $\Delta_1(E_i)\le \delta_i$ and $\Delta_1(A_i)\le 1+\delta_i$.
    Consider the expansion~\eqref{common_expansion} and define
\[
 C^{(m)}:=
 \sum_{|S|=m} T_S
 = \sum_{|S|=m}
 \bigotimes_{i\in S}A_i\otimes
 \bigotimes_{j\notin S}E_j.
\]
Thus $\mathcal{F}_G=\sum_{m=0}^dC^{(m)}$.

    For each $i$ Lemma~\ref{lem:local-spectral} gives a common orthonormal eigenbasis of
    $A_i$ and $E_i$. Consider the tensor-product eigenbasis~\eqref{tensor_eigenbasis}.
    For vectors of this basis we denote
\[
    \mathrm{NZ}(v):=\{i:A_i v^{(i)}_{j_i}\ne0\},
 \qquad
 n(v):=|\mathrm{NZ}(v)|,
\]
and let $e(v)$ be the product of the corresponding eigenvalues of
$E_1,\ldots,E_d$.
Then~\eqref{eq:local-dichotomy} gives
\begin{equation}
 \label{eq:balanced-layer-action}
 \mathcal{F}_Gv=e(v)(-1)^{n(v)}v,
 \qquad
 C^{(m)}v=e(v)(-2)^m\binom{n(v)}m v
\end{equation}
(indeed, the sum in $C^{(m)}v$ is over $m$-sets $S\subset\mathrm{NZ}(v)$).

    For a randomly (and uniformly) chosen tensor basis
    vector~\eqref{tensor_eigenbasis}, the quantity $n(v)$ is a sum of independent
    Bernoulli variables with means at most $1/2$. Let $\mu:=\E n(v)\le d/2$.
    Choose an interval $I=\{a,a+1,\ldots,a+k-1\}$ of $k$
    consecutive integers centered at $\mu$. We ensure that $a\ge 0$ by
    shifting $I$ if necessary. Hoeffding's inequality gives
\[
 \PP\{n(v)\notin I\}
 \le2\exp\left(-c\frac{k^2}{d}\right).
\]
    If $I=\{a,\ldots,a+k-1\}$, then the polynomial
\[
 Q(x):=(-1)^{a+1}\sum_{j=0}^{k-1}(-2)^j\binom{x-a}{j}
\]
satisfies $Q(n)=(-1)^{n+1}$ on $I$.
Write
$$
 Q(x)=\sum_{m=0}^{k-1}h_m(-2)^m\binom xm
$$
with some coefficients $h_0,\ldots,h_{k-1}\in\mathbb{R}$.
Define the approximation matrix as
\[
 B_G:=\mathcal{F}_G+\sum_{m=0}^{k-1}h_mC^{(m)}.
\]
It vanishes on every basis vector with $n(v)\in I$ due to~\eqref{eq:balanced-layer-action}.  Hence
\[
 \rank B_G
 \le N\,\PP\{n(v)\notin I\}
 \le N\cdot 2\exp\left(-c\frac{k^2}{d}\right).
\]

It remains to bound the approximation error.
We have
$$
\Delta_1(\mathcal{F}_G-B_G) \le k \max_{m<k}|h_m| \cdot \max_{m<k}\Delta_1(C^{(m)}).
$$

First, we bound the coefficients $h_m$.
Vandermonde's convolution identity (see, e.g.,~\cite[Sec.~5.1]{Concrete}) gives
\[
 \binom{x-a}{j}
 =\sum_{m=0}^j\binom xm\binom{-a}{j-m}.
\]
(Note the convention $\binom{x}{j}=0$ for integer $j<0$.)
Consequently,
$$
 h_m=\frac{(-1)^{a+1}}{(-2)^m}
 \sum_{j=0}^{k-1}(-2)^j\binom{-a}{j-m}
 = (-1)^{a+1} \sum_{r=0}^{k-1-m} (-2)^r \binom{-a}{r}.
$$
We have
\[
 \left|\binom{-a}{r}\right|
 =\binom{a+r-1}{r}
 \le\binom dr,
\]
because $r<k$ and $a+k-1\le d$.  Since $k\le d/2$,
\[
 |h_m|
 \le\sum_{r=0}^{k-1} 2^r\binom dr
 \le 2^k (ed/k)^k = (2ed/k)^k.
\]

Second, we bound the quantities $\Delta_1(C^{(m)})$.
Put $u_i:=(1+\delta_i)/\delta_i$,
$$
u_i = 1 + \delta_i^{-1} \le 1 + (|G_i|/2)^{1/2} \le 2(N_{\max})^{1/2}.
$$
We have
$$
\Delta_1(C^{(m)})
\le \sum_{|S|=m} \prod_{i\in S}(1+\delta_i) \cdot \prod_{j\not\in S}\delta_j
= \prod_{i=1}^d \delta_i \cdot \sum_{|S|=m} \prod_{i\in S} u_i
= \prod_{i=1}^d \delta_i \cdot \sigma_m(u_1,\ldots,u_d),
$$
where $\sigma_m$ is the elementary symmetric polynomial of degree $m$. Note that
$$
m!\sigma_m(x_1,\ldots,x_d) \le (\sum_{i=1}^d x_i)^m,
\quad x_1,\ldots,x_d\ge 0,
$$
since every square-free monomial of degree $m$ occurs $m!$ times in the right
side of the inequality. Using that $m!\ge(m/e)^m$, we obtain
$$
\sigma_m(u_1,\ldots,u_d)
\le \frac{1}{m!}(\sum_{i=1}^d u_i)^m
\le (e/m)^m (2d N_{\max}^{1/2})^m
= (2ed N_{\max}^{1/2} / m)^m
\le (2ed N_{\max}^{1/2} / k)^k.
$$
(For $m=0$ we have $\sigma_0=1$.)
By the definition of $\delta_i$, we have
$$
\prod_{i=1}^d \delta_i \le 2^d N^{-1/2}.
$$

We put all the inequalities altogether and obtain
$$
\Delta_1(\mathcal{F}_G-B_G)
\le k(2ed/k)^k \cdot 2^d N^{-1/2} (2ed N_{\max}^{1/2}/k)^k
\le N^{-1/2} 2^d (C d^2/k^2)^k (N_{\max})^{k/2}.
$$

\end{proof}

\begin{lemma}[Balanced products]
\label{lem:fixed-balanced}
For every $\eta>0$ there are $L_0=L_0(\eta)$ and
$d_0=d_0(\eta)$ such that,
for every fixed $L\ge L_0$ there is a constant $a=a(L,\eta)>0$ with the
following property. If
\[
 G=G_1\times\cdots\times G_d,
 \qquad
 L\le|G_i|\le L^2,
 \qquad
 d\ge d_0,
\]
then there is a matrix $B_G$ satisfying
\[
 \rank B_G\le N^{1-a},
 \qquad
    \Delta_1(\mathcal{F}_G-B_G)\le N^{-1/2+\eta},
    \qquad N:=|G|.
\]
\end{lemma}

\begin{proof}
In Lemma~\ref{lem:balanced-products}, take
    $k=\lfloor\theta d\rfloor$, where $\theta=\theta(\eta)$ is sufficiently small.
    The estimate~\eqref{eq:balanced-general-error} gives:
$$
    \frac{\log N^{1/2} \Delta_1(\mathcal{F}_G-B_G)}{\log N}
    \le \frac{d\log 2+ k\log(Cd^2/k^2) + k/2 \log L^2}{d\log L}
    \lesssim \frac{1}{\log L_0} + \frac{\theta\log(1/\theta)}{\log L_0} + \theta.
$$
Choose $\theta$ sufficiently small and $L_0$ sufficiently
    large so that this is less than $\eta$; this ensures the bound
    $N^{-1/2+\eta}$ for the error.

Finally, from rank estimate~\eqref{eq:balanced-general-rank} we obtain:
$\rank B_G\le N\cdot 2e^{-c\theta^2d} \le N^{1-a}$, for $a\asymp\theta^2/\log L$.
\end{proof}

\subsection{A near-optimal cyclic block}

Here we consider primorial numbers, i.e. products of all primes not exceeding
given thresholds. We will use letter $q$ to denote primes.

\begin{lemma}[Primorial block]
\label{lem:primorial-block}
Let $\eta>0$.  There is a constant
$c=c(\eta)>0$ such that for every sufficiently large primorial
\[
 D=\prod_{q\le y}q
\]
the Fourier matrix admits a decomposition $\mathcal{F}_D=B_D+E_D$ with
\begin{equation}
 \label{eq:primorial-rho}
 \rank B_D\le D^{1-c/\log\log D}
\end{equation}
and
\begin{equation}
 \label{eq:primorial-delta}
 \Delta_1(E_D) \le D^{-1/2+\eta}.
\end{equation}
\end{lemma}

\begin{proof}
By the Chinese remainder theorem,
\[
 \Z_D\cong\prod_{q\le y}\Z_q.
\]
Put $d:=\pi(y)$ and choose $k:=\lfloor\theta d\rfloor$,
where $0<\theta<1/2$ will depend only on $\eta$.  Apply
Lemma~\ref{lem:balanced-products} to the factors $\Z_q$, $q\le y$. It gives the
decomposition $\mathcal{F}_D = B_D+E_D$ with
\[
 \rank B_D
 \le D\cdot 2e^{-c k^2/d}
\]
and
\begin{equation}
\label{approx_error_estimate_D}
    \Delta_1(E_D) \le D^{-1/2}2^d (Cd^2/k^2)^k y^{k/2}.
\end{equation}

Standard prime-number estimates give
\[
 d\asymp\frac{y}{\log y},
 \qquad
 \log D\asymp y,
 \qquad
 \log\log D\asymp\log y.
\]

We substitute in the approximation error
estimate~\eqref{approx_error_estimate_D} the choice $k=\lfloor \theta d\rfloor$
and the order estimates for $d$ and $D$ in terms of $y$:
$$
    \frac{\log D^{1/2}\Delta_1(E_D)}{\log D}
    \lesssim \frac{d + \theta d \log(1/\theta^2) + \theta d \log y}{\log D}
    \lesssim \frac{1}{\log y} + \theta.
$$
Choose $\theta=\theta(\eta)>0$ sufficiently small and ensure that $y$ is
    sufficiently large to make the bound less than $\eta$. This gives
    $\Delta_1(E_D)\le D^{-1/2+\eta}$.

The rank estimate gives
\[
\rank B_D \le D\cdot 2e^{-c_1\theta^2d} \le D^{1-c_2/\log\log D}.
\]
\end{proof}

\subsection{Passage from $\Z_D^m$ to $\Z_{D^m}$}

The primorial orders lemma is not enough to obtain a power gain in rank.
We will need to fix a ``block'' $\Z_D$, consider groups $\Z_D^m$,
$m\to\infty$, and relate the structure of $\mathcal{F}_{D^m}$ to that of
$\mathcal{F}_{\Z_D^m}=\mathcal{F}_D^{\otimes m}$.

We identify the row and column indices of $\mathcal{F}_{D^m}$ with their base-$D$
digit vectors by writing
\begin{equation}
    \label{digit_vectors}
 x=\sum_{i=0}^{m-1}x_iD^i,
 \qquad
 y=\sum_{j=0}^{m-1}y_jD^{m-1-j},
 \qquad
 0\le x_i,y_j<D.
\end{equation}
We have
$$
\mathcal{F}_{D^m}(x,y) = \exp(2\pi i xy/D^m),
\qquad \mathcal{F}_D^{\otimes m}(x,y) = \exp(2\pi i\sum_{j=0}^{m-1}x_jy_j/D).
$$

Thus $\mathcal{F}_{D^m}$ is reindexed by the same digit pairs that index the tensor
power $\mathcal{F}_D^{\otimes m}$ (with the column digits written in reverse order).
The next lemma shows that, after this reindexing, the difference between the two
Fourier matrices is only a low-complexity phase correction.

For entrywise products we use
\begin{equation}
 \label{eq:Hadamard-rank}
 \rank(A\circ B)\le\rank A\,\rank B,
 \qquad
    \Delta_1(A\circ B)\le\Delta_1(A)\Delta_\infty(B).
\end{equation}

\begin{lemma}[Phase correction]
\label{lem:approx-tensorization}
There is an absolute constant $C>0$ such that the following holds.
For every sufficiently large $D$ and $m$ there is a matrix $\Psi_{D,m}$ of order $D^m$
such that
\[
 \rank \Psi_{D,m}\le\exp(Cm\log\log D),
 \qquad
    \Delta_\infty(\Psi_{D,m})\le2,
\]
and
\[
    \Delta_\infty(\mathcal{F}_{D^m}-\mathcal{F}_D^{\otimes m}\circ \Psi_{D,m})\le D^{-m}.
\]
\end{lemma}

\begin{proof}
    Using the notation from~\eqref{digit_vectors}, we write
$$
    \frac{xy}{D^m} = \sum_{i,j=0}^{m-1} x_iy_j D^{i-j-1}.
$$
    The terms with $i>j$ are integers and the terms with $i=j$ give
    $\mathcal{F}_D^{\otimes m}$. Hence
\[
 \mathcal{F}_{D^m}(x,y)=\mathcal{F}_D^{\otimes m}(x,y)e^{2\pi i\Phi(x,y)},
\]
where
\[
 \Phi(x,y) :=\sum_{0\le i<j\le m-1}\frac{x_iy_j}{D^{j-i+1}}
 =\sum_{i=0}^{m-2}u_i(x)v_i(y),
\]
\[
 u_i(x):=x_i,
 \qquad
 v_i(y):=\sum_{j>i}\frac{y_j}{D^{j-i+1}}.
\]
Since $v_i(y)\le1/D$, one has $0\le\Phi(x,y)<m$.

Let
\[
 P_r(t):=\sum_{\ell=0}^r\frac{(2\pi it)^\ell}{\ell!},
 \qquad
 r:=\lceil m\log D\rceil.
\]
Define $\Psi_{D,m}(x,y):=P_r(\Phi(x,y))$.  Taylor's remainder and Stirling's
lower bound give
\[
 \left|e^{2\pi i\Phi}-P_r(\Phi)\right|
 \le\frac{(2\pi m)^{r+1}}{(r+1)!}
 \le\left(\frac{2\pi e}{\log D}\right)^{m\log D}
 \le D^{- m}
\]
for sufficiently large $D$.  This also gives
    $\Delta_\infty(\Psi_{D,m})\le2$.

By the multinomial expansion, $\Psi_{D,m}$ is a sum of rank-one
terms:
\[
\Psi_{D,m}(x,y)
 = P_r\left(\sum_{i=0}^{m-2}u_i(x)v_i(y)\right)
 = \sum_{\nu_0+\ldots+\nu_{m-2}\le r}
 c_\nu
 \left(\prod_{i=0}^{m-2} u_i(x)^{\nu_i}\right)
 \left(\prod_{i=0}^{m-2} v_i(y)^{\nu_i}\right).
\]
The terms are indexed by
$\nu_0+\cdots+\nu_{m-2}\le r$.
Consequently,
\[
 \rank \Psi_{D,m}
 \le\binom{m-1+r}{m-1}
 \le (C\log D)^m
 \le\exp(Cm\log\log D).
\]
\end{proof}

\subsection{Higher-order Fourier-grid interpolation}

The last ingredient we need is the way
of passing from a low-rank approximation of higher-order Fourier matrix $\mathcal{F}_M$ to
an approximation of lower-order matrix $\mathcal{F}_N$. The simplest way is to pick first
$N$ rows from the larger approximant and pick columns with indices
$m_j:=\lfloor jM/N\rfloor$, using that
$$
\mathcal{F}_M(a,m_j) = \exp(2\pi iam_j/M)\approx \exp(2\pi iaj/N) = \mathcal{F}_N(a,j),\quad 0\le a,j<N.
$$
Unfortunately, this gives poor approximation error and we use a more involved
method, the polynomial interpolation.

Let us recall basic properties of Lagrange interpolation.  Let
$\ell_0(x),\ldots,\ell_{s-1}(x)$ be the Lagrange basis for some distinct nodes
$x_0,\ldots,x_{s-1}\in[a,b]$:
\begin{equation}
\label{lagrange_polynomial}
    \ell_k(x) = \prod_{\substack{0\le r< s\\r\ne k}}\frac{x-x_r}{x_k-x_r}.
\end{equation}
A function $f$ is interpolated at these nodes by the polynomial $Lf(x) :=
\sum_{r=0}^{s-1} f(x_r)\ell_r(x)$. We will use a standard interpolation
remainder formula, valid for real-valued functions $f\in C^s[a,b]$ and all $x\in[a,b]$:
$$
f(x) - Lf(x) = \frac{f^{(s)}(\xi)}{s!}\prod_{r=0}^{s-1}(x-x_r),
\quad \xi\in[a,b].
$$
Then the inequality for complex-valued functions follows:
\begin{equation}
    \label{interpolation_bound}
    |f(x) - Lf(x)| \le \frac{2}{s!}\max_{a\le \xi\le b}|f^{(s)}(\xi)|\cdot \prod_{r=0}^{s-1}|x-x_r|.
\end{equation}

\begin{lemma}[Higher-order Fourier-grid interpolation]
\label{lem:fourier-grid-interpolation}
Let $M>N$ and let $2\le s\le\lfloor M/N\rfloor$.  There are a row-selection
matrix $R\in\{0,1\}^{N\times M}$ and a matrix $P\in\C^{M\times N}$ such that
\begin{equation}
 \label{eq:grid-deterministic}
    \Delta_\infty(\mathcal{F}_N-R\mathcal{F}_MP)
 \le2\left(\frac{2\pi N}{M}\right)^s,
\end{equation}
while, for every $M\times M$ matrix $E$,
\begin{equation}
 \label{eq:grid-transfer}
 \Delta_1(REP)
    \le2^s\frac{M}{N}\Delta_1(E).
\end{equation}
\end{lemma}

\begin{proof}
Let $R$ select the first $N$ rows.  Put
\[
 m_j:=\lfloor Mj/N\rfloor,
 \qquad
 \tau_j:=Mj/N-m_j\in[0,1).
\]
Let $\ell_0,\dots,\ell_{s-1}$ be the Lagrange basis for the nodes
    $x_0=0$, $x_1=1$, \dots, $x_{s-1}=s-1$, and define for $j=0,\ldots,N-1$ the $j$th column of $P$ by
\[
 P_{m_j+r,j}:=\ell_r(\tau_j),
 \qquad r=0,1,\ldots,s-1.
\]
    Note that $m_{N-1}+s-1\le M(N-1)/N+\lfloor M/N\rfloor-1\le M-1$, so $P$ is
    correctly defined. All other entries of $P$ are zero. We will also use that
    the supports of columns of $P$ are distinct since $m_{j+1}-m_j\ge\lfloor
    M/N\rfloor\ge s$.

    Let us estimate the error $\Delta_\infty(\mathcal{F}_N-R\mathcal{F}_MP)$. Fix $0\le a,j<N$ and
consider the function $f(x):=\exp(2\pi ia(m_j+x)/M)$. We have
$$
f(\tau_j)
= \exp(2\pi ia(m_j+\tau_j)/M)
= \exp(2\pi iaj/N)
= (\mathcal{F}_N)_{a,j}
$$
and
$$
(R\mathcal{F}_MP)_{a,j} = \sum_{r=0}^{s-1}\exp(2\pi ia(m_j+r)/M)\ell_r(\tau_j) =
Lf(\tau_j).
$$
    Using the interpolation error bound~\eqref{interpolation_bound}, we obtain
    $$
    |(\mathcal{F}_N)_{a,j}-(R\mathcal{F}_MP)_{a,j}|
    \le \frac{2}{s!}\max|f^{(s)}(x)| \cdot\prod_{r=0}^{s-1}|\tau_j-r|
    \le \frac{2}{s!}\left(\frac{2\pi a}{M}\right)^s (s-1)!
    \le 2\left(\frac{2\pi N}{M}\right)^s,
    $$
which proves~\eqref{eq:grid-deterministic}.

Let us prove~\eqref{eq:grid-transfer}.
For $\tau\in[0,1]$ we have:
$$
    |\ell_r(\tau)|
    = \prod_{\substack{j\ne r\\0\le j<s}}\frac{|\tau-j|}{|r-j|}
    \le \frac{(s-1)!}{\prod_{j\ne r}|r-j|}
    = \frac{(s-1)!}{r!(s-1-r)!} = \binom{s-1}{r},
$$
hence
$$
    \sum_{r=0}^{s-1} |\ell_r(\tau)| \le 2^{s-1} \le 2^s.
$$
For a fixed $0\le a,j<N$ we have
$$
    |(REP)_{a,j}|
    \le \sum_{r=0}^{s-1} |\ell_r(\tau_j)|\,|E_{a,m_j+r}|.
$$
Summing over $j$, using the disjointness of the supports and the inequality
    $|\ell_r(\tau)|\le 2^s$, we obtain
$$
\sum_{j=0}^{N-1}|(REP)_{a,j}| \le 2^{s} \sum_{m=0}^{M-1}|E_{a,m}|.
$$
The estimate~\eqref{eq:grid-transfer} follows.
\end{proof}

\subsection{The cyclic case}

\begin{proposition}[DFT approximation]
\label{prop:cyclic-polynomial-rank}
Let $\eta>0$.  There are constants
$a=a(\eta)>0$ and $N_0=N_0(\eta)$ such that, for every $N\ge N_0$,
there is a matrix $B_N$ satisfying
\[
 \rank B_N\le N^{1-a},
 \qquad
 \Delta_1(\mathcal{F}_N-B_N)\le N^{-1/2+\eta}.
\]
\end{proposition}

\begin{proof}
We may assume $0<\eta<1/2$.  Choose $\eta_0>0$ and
$\theta\in(0,1)$ so that
\begin{equation}\label{eq:theta-choice}
 \theta(1/2-\eta_0)>1/2-\eta/8.
\end{equation}
Take a sufficiently large primorial $D$ from
Lemma~\ref{lem:primorial-block}, applied with parameter $\eta_0$, and write
\[
 \mathcal{F}_D=B_D+E_D,
 \qquad
 \rho_D:=\frac{\rank B_D}{D},
 \qquad
 \delta_D:=\Delta_1(E_D).
\]

For $m$ large put $s=\lfloor\theta m\rfloor$ and apply
Lemma~\ref{lem:tensor-truncation} to $\mathcal{F}_D^{\otimes m}$.  Its rank estimate
gives
$$
    \frac{\rank A_m}{D^m}
    \le 2^m \rho_D^{m-s} \le 2^m D^{-m(1-\theta)c/\log\log D}.
$$
Let $K_D:=C\log\log D$ be the exponent in the rank bound of
Lemma~\ref{lem:approx-tensorization}. If $D$ is sufficiently large, we can
ensure that
$$
    \frac{\rank A_m}{D^m} \le \exp(-m(K_D+1)).
$$
The error estimate is
\[
 \Delta_1(\mathcal{F}_D^{\otimes m}-A_m)
 \le 4^m \delta_D^{s+1}
    \le 4^m D^{m\theta(-1/2+\eta_0)}
    = (D^m)^{-\theta(1/2-\eta_0)+\frac{\log 4}{\log D}}
 \le (D^m)^{-1/2+\eta/4}
\]
for sufficiently large $D$. Now we can fix $D$; it depends only on 
$\eta$.

Let $\Psi_{D,m}$ be the phase correction from
Lemma~\ref{lem:approx-tensorization} and put
$\widetilde B_{D^m}:=A_m\circ \Psi_{D,m}$.  Since the rank of a Hadamard
product is at most the product of the ranks,
\[
 \rank\widetilde B_{D^m}
    \le D^me^{-m(K_D+1)}\cdot e^{K_Dm}
 = D^me^{-m}
 =(D^m)^{1-a_0},
 \qquad
 a_0:=\frac{1}{\log D}.
\]
Also,
    \begin{multline*}
 \Delta_1(\mathcal{F}_{D^m}-\widetilde B_{D^m})
 \le \Delta_1(\mathcal{F}_{D^m} - \mathcal{F}_D^{\otimes m}\circ\Psi_{D,m}) + \Delta_1(\mathcal{F}_D^{\otimes m}\circ\Psi_{D,m} - \widetilde{B}_{D^m}) \le \\
        \le D^{-m}+\Delta_1(\mathcal{F}_D^{\otimes m}-A_m)\Delta_\infty(\Psi_{D,m})
    \le D^{-m} + 2(D^m)^{-1/2+\eta/4}
 \le(D^m)^{-1/2+\eta/3}
 \end{multline*}
for all sufficiently large $m$. Thus we obtained the required approximation for the powers of $D$.

It remains to pass to an arbitrary $N$.  Let $M=D^m$ be
the least power of $D$ with
\[
 M\ge N\log N.
\]
Then $M<DN\log N$.  Choose a constant
\[
 0<c_\eta<\min\left\{\frac12,\frac{\eta}{4\log2}\right\}
\]
and put $s=\lfloor c_\eta\log N\rfloor$.  For all sufficiently large $N$,
$2\le s\le M/N$, so Lemma~\ref{lem:fourier-grid-interpolation} applies to
$\widetilde B_M$ and interpolation order $s$. The resulting matrix
$B_N:=R\widetilde{B}_MP$ satisfies
\[
    \rank B_N\le \rank \widetilde{B}_M \le M^{1-a_0}
 \le(DN\log N)^{1-a_0}
 \le N^{1-a_0/2}.
\]
Let us bound the error:
\begin{multline*}
\Delta_1(\mathcal{F}_N-B_N)
\le \Delta_1(\mathcal{F}_N-R\mathcal{F}_MP) + \Delta_1(R\mathcal{F}_MP -
    R\widetilde{B}_MP) \le \\
 \le2\left(\frac{2\pi N}{M}\right)^s +2^s \frac MN M^{-1/2+\eta/3} \le \\
 \le2\left(\frac{2\pi}{\log N}\right)^s +N^{c_\eta\log2}(D\log N)(N\log N)^{-1/2+\eta/3}
 \le N^{-1/2+\eta}
\end{multline*}
for all sufficiently large $N$.  This proves the proposition.
\end{proof}

\subsection{Passage to arbitrary finite abelian groups}

\begin{proof}[Proof of Theorem~\ref{thm:main-rank}]
Apply Proposition~\ref{prop:cyclic-polynomial-rank} with a parameter $\eta_1>0$ so small that
\[
 b_1:=1/2-\eta_1>1/2-\eta/4.
\]
Thus every sufficiently large cyclic Fourier matrix $\mathcal{F}_N$ has an approximation
with rank at most $N^{1-a_1}$ and error at most $N^{-b_1}$.

    Choose $\sigma\in(0,1)$ and $\zeta>0$ so small that
\[
 b_1(1-\sigma)-\zeta>1/2-\eta/3.
\]
Let $L$ be large enough that Proposition~\ref{prop:cyclic-polynomial-rank}
applies to every $N\ge L$, Lemma~\ref{lem:polynomial-tensor-closure}
applies with parameters $a_1,b_1,\sigma,\zeta$, and
$L\ge L_0(\eta/3)$ from Lemma~\ref{lem:fixed-balanced}. Note that $L$ depends only on $\eta$.

Write
\[
 G=\Z_{N_1}\times\cdots\times\Z_{N_s}.
\]
Let $G_{\rm large}$ be the product of all factors with $N_i\ge L$.
Group the remaining cyclic factors into blocks
$G^{(1)},\dots,G^{(d)}$ satisfying
\[
    L\le|G^{(j)}|<L^2,
\]
    with one remainder $G_{\rm rem}$ of order less than $L$. If $d\ge d_0(\eta/3)$ from Lemma~\ref{lem:fixed-balanced}, we
    consider ``balanced'' product $G_{\rm bal}:=\prod_{j=1}^d G^{(j)}$;
    otherwise, we absorb the blocks $G^{(j)}$ into $G_{\rm rem}$. The latter case is simpler and we will not consider it further. So,
    $$
    G=G_{\rm rem} \times G_{\rm bal} \times G_{\rm large}.
    $$

    We approximate $\mathcal{F}_{G_{\rm bal}}$ using Lemma~\ref{lem:fixed-balanced} by a matrix of rank $|G_{\rm bal}|^{1-a_2}$
    and the error at most $|G_{\rm bal}|^{-1/2+\eta/3}$.

    For every factor in $G_{\rm large}$ we use Proposition~\ref{prop:cyclic-polynomial-rank} to approximate $\mathcal{F}_{N_i}$ by a
    matrix of rank $N_i^{1-a_1}$ and error $N_i^{-b_1}$.
    Combining all the approximations by
    Lemma~\ref{lem:polynomial-tensor-closure} gives an approximation for $\mathcal{F}_{G_{\rm large}}$ with rank $|G_{\rm
    large}|^{1-a_3}$ and error at most
    $$
    |G_{\rm large}|^{-b_1(1-\sigma)+\zeta}
    \le |G_{\rm large}|^{-1/2+\eta/3}.
    $$

    Put $b_2:=1/2-\eta/3$ and choose $\sigma_2,\zeta_2>0$ so small that
    $b_2(1-\sigma_2)-\zeta_2>1/2-\eta/2$.

    Apply Lemma~\ref{lem:polynomial-tensor-closure} once more now to $G_{\rm
    bal}$ and $G_{\rm large}$, with $a=\min\{a_2,a_3\}$, $b_2$, and parameters
    $\sigma_2,\zeta_2$. 
    If any of two factors is not large enough to apply Lemma, we can absorb it
    into $G_{\rm rem}$.
    We obtain, for $G':=G_{\rm bal}\times G_{\rm large}$, a matrix $B_{G'}$ such
    that
    $$
    \rank B_{G'}\le |G'|^{1-a_4},
    \qquad
    \Delta_1(\mathcal{F}_{G'}-B_{G'}) \le |G'|^{-1/2+\eta/2},
    $$

    Finally, tensor with the exact Fourier matrix $\mathcal{F}_{G_{\rm rem}}$, i.e. put
    $B_G:=B_{G'}\otimes \mathcal{F}_{G_{\rm rem}}$. The order of $G_{\rm rem}$ is
    bounded in terms of $\eta$ only. This makes only a harmless
    decrease of the final exponent $a>0$ and an increase of the threshold
    $N_0$.
\end{proof}

\section{Consequences}

\subsection{Consecutive trigonometric frequencies}

We denote $e_k(x):=\exp(2\pi ikx)$. Let $\mathcal{T}_N:=\Span\{e_k\}_{k=-N}^N$ be the space of
trigonometric polynomials of order $N$, and
$D_N$ be the Dirichlet kernel, $D_N(x):=\sum_{|k|\le N}e_k(x)$. We
will use the standard trigonometric interpolation operator
$$
P_N\colon\mathbb{C}^{2N+1}\to\mathcal{T}_N,
\qquad
(P_N y)(x) = \frac{1}{2N+1}\sum_{k=0}^{2N} y_k D_N\bigl(x-\frac{k}{2N+1}\bigr).
$$
The operator $P_N$ interpolates functions on the equidistant grid $x_k=k/(2N+1)$,
$k=0,\ldots,2N$. Namely, if $f$ is a function on $\T$ and
$y=(f(x_0),\ldots,f(x_{2N}))$, then $(P_Ny)(x_k)=f(x_k)$ for all $k$. Therefore,
if $f\in\mathcal{T}_N$, then $P_Ny\equiv f$.

Standard bounds on the Dirichlet kernel,
$\|D_N\|_1\lesssim\log N$ and $|D_N(x)|\lesssim\min\{N,|x|^{-1}\}$, $|x|\le1/2$,
imply that
$$
\|P_N\|_{L_p^{2N+1}\to L_p(T)}\lesssim\log N
$$
for $p=1$ and $p=\infty$.
The same bound holds for all $p\in(1,\infty)$ by interpolation, and this will be
sufficient for our purposes.
(In fact, Marcinkiewicz inequality between $L_p$-norms of trigonometric polynomials on
$\T$ and on the equidistant grid gives $\|P_N\|_{p\to p}\le C(p)$,
$p\in(1,\infty)$, but it is more convenient for us to use $p$-free estimate.)

\begin{proposition}[Consecutive trigonometric frequencies]
\label{prop:consecutive-frequencies}
Let $1\le p<2$ and $\eta>0$.  There are constants
$a=a(\eta)>0$ and $N_0=N_0(\eta)$ such that, for every
$N\ge N_0$,
\[
 d_{N^{1-a}}^{\mathbb C}
    \bigl(\{e^{2\pi ikx}\}_{k=-N}^{N},L_p(\T)\bigr)
 \le N^{-\alpha_p+\eta}.
\]
\end{proposition}

\begin{proof}
Apply Theorem~\ref{thm:main-rank} with parameter $\eta/2$ to the matrix $\mathcal{F}_{2N+1}$.
The resulting approximant matrix $B$ has rank $n\lesssim N^{1-a}$ and 
the $L_p$-approximation error
$$
    \Delta_p(\mathcal{F}_{2N+1}-B) = (2N+1)^{-1/p}\|\mathcal{F}_{2N+1}-B\|_{p,\infty}
    \le (2N+1)^{-1/p}\|\mathcal{F}_{2N+1}-B\|_{1,\infty}
    \lesssim N^{-\alpha_p+\eta/2}.
$$

It is convenient to index rows of $\mathcal{F}_{2N+1}$ by numbers $k=-N,\ldots,N$.
The row with index $k$ is the
vector $y^k:=(e_k(\frac{j}{2N+1}))_{j=0}^{2N}$. Since $P_N$ recovers
$\mathcal T_N$, we have $P_Ny^k = e_k$. Therefore,
\begin{multline*}
    d_n^{\mathbb{C}}(\{e_{-N},\ldots,e_{N}\}, L_p(\T)) 
    \le \|P_N\|_{p\to p}\, d_n^{\mathbb{C}}(\{y^{-N},\ldots,y^{N}\}, L_p^{2N+1}) \le \\
    \le \|P_N\|_{p\to p}\, \Delta_p(\mathcal{F}_{2N+1}-B)
    \lesssim (\log N) N^{-\alpha_p+\eta/2}.
\end{multline*}
The required bound follows.
\end{proof}

\subsection{Centrally symmetric convex trigonometric spectra}

Fix an integer $d\ge1$.  For a finite set $\Lambda\subset\Z^d$ write
\[
 \mathcal E_\Lambda
 :=\{e_\lambda:\lambda\in\Lambda\},
 \qquad
 e_\lambda(x):=e^{2\pi i\langle\lambda,x\rangle},
 \quad x\in\T^d.
\]

We will apply the discrete John's theorem of Tao and Vu~\cite{TaoVu2008}.
Classic John's theorem states that any symmetric convex
body in $\R^d$ is, roughly speaking, an ellipsoid up to some constant; the discrete version describes an
intersection of a convex body with a lattice as a symmetric generalized
arithmetic progression (GAP). Here we restrict ourselves to the lattice
$\mathbb{Z}^d$. 
A \emph{symmetric GAP} in $\mathbb{Z}^d$ is defined by its ``rank'' $r$, the ``dimensions''
$N_1,\ldots,N_r$ and ``steps'' $v_1,\ldots,v_r\in\mathbb{Z}^d$.
We will identify a symmetric GAP with its ``image''~--- the set
$$
P = \{n_1v_1+\ldots+n_rv_r\colon
n_j\in\mathbb{Z},\;|n_j|\le N_j,\;j=1,\ldots,r\}.
$$
Tao and Vu work with the dimensions $N_j$ that are positive real numbers, in order to
define proper scaling of GAPs. For us it would be convenient to assume that
$N_j\in\mathbb{Z}_+$.

Formally, we need the following corollary of~\cite[Theorem~1.6]{TaoVu2008}.

\begin{exttheorem}
\label{thm_discrete_john}
    Let $K$ be a convex symmetric body in $\R^d$.
    There exists a symmetric GAP in $\mathbb{Z}^d$,
    $P=\{n_1v_1+\ldots+n_rv_r\colon |n_j|\le N_j\}$, such that $r\le
    d$,
    $$
    K\cap\mathbb{Z}^d \subset P,
    \qquad |K\cap\mathbb{Z}^d| \ge c(d)|P|,
    $$
    and the steps $v_1,\ldots,v_r$ are linearly independent.
\end{exttheorem}

For a symmetric GAP with linearly independent steps (it is also called
\emph{infinitely proper}) we have $|P|=\prod(2 N_j+1)$.

Now we are ready to prove Proposition~\ref{prop:intro-convex-spectrum}.
\begin{proof}
    First we consider a parallelopiped spectra
    $$
    \Lambda=\Pi(N_1,\ldots,N_d):=\prod_{j=1}^d \{-N_j,\ldots,N_j\},
    \qquad N = |\Lambda| = \prod_{j=1}^d (2N_j+1).
    $$
    The case $d=1$ was settled in
    Proposition~\ref{prop:consecutive-frequencies}, the construction for $d>1$
    is analogous. We consider the multidimensional interpolation,
    $$
    P_{N_1,\ldots,N_d} y = \frac{1}{\prod_{j=1}^d (2N_j+1)}\sum_{k_1=0}^{2N_1}\cdots\sum_{k_d=0}^{2N_d}
    y_{k_1,\ldots,k_d} \prod_{j=1}^d D_{N_j}(x_j-k_j/(2N_j+1)).
    $$
    Similarly to the univariate case we check that
    $\|P_{N_1,\ldots,N_d}\|_{p\to p}\lesssim (\log N)^d$.
    We will apply Theorem~\ref{thm:main-rank} to the group
    $$
    G := \prod_{j=1}^d \mathbb{Z}_{2N_j+1}
    $$
    and the parameter $\eta/2$. As in the univariate case, if $y$ is a row of
    $\mathcal{F}_G$ indexed by numbers $(k_1,\ldots,k_d)$, $|k_j|\le N_j$, then
    $P_{N_1,\ldots,N_d}$ sends $y$ to $e_{\vec{k}}$, $\vec{k}=(k_1,\ldots,k_d)$.
    Therefore, for $n=N^{1-a(\eta/2)}$ we have
    $$
    d_n^{\mathbb{C}}(\{e_\lambda\colon\lambda\in\Lambda\}, L_p(\T^d)) 
    \le \|P_{N_1,\ldots,N_d}\|_{p\to p}\, \min_{\rank B\le n}
    \Delta_p(\mathcal{F}_G-B)
    \lesssim (\log N)^d N^{-\alpha_p+\eta/2}.
    $$
    This settles the first case.

    Now suppose that $\Lambda$ is a symmetric GAP in $\mathbb{Z}^d$ of rank
    $r\le d$, i.e.,
    $$
    \Lambda = P = \{n_1v_1+\ldots+n_rv_r\colon |n_j|\le N_j\},
    $$
    and the steps $v_j\in\mathbb{Z}^d$ are linearly independent.
    Consider the map
\[
 \theta\colon\T^d\to\T^r,
 \qquad
 \theta(x):=(\langle v_1,x\rangle,\dots,\langle v_r,x\rangle)\pmod1,
\]
    and the pullback operator $(\Theta f)(x):=f(\theta(x))$. 
    Since $v_j$ are linearly independent, $\theta$ is surjective. Moreover, this
    is a homomorphism, hence it preserves measure.
    Therefore, $\Theta$ is an isometry from $L_p(\T^r)$ to $L_p(\T^d)$.
    Take $\vec{k}:=(k_1,\ldots,k_r)$. We have
    $$
    \Theta e_{\vec{k}}(x_1,\ldots,x_r) = e_{\vec{k}}(\theta(x))
    = \exp(2\pi i\,\sum_{s=1}^r k_s \langle v_s,x\rangle)
    = e_\lambda(x),
    \quad \lambda = \sum_{s=1}^r k_s v_s,
    $$
    so,
    $$
    d_n^{\mathbb{C}}(\mathcal{E}_P,L_p(\T^d))
    \le d_n^{\mathbb{C}}(\mathcal{E}_{\Pi(N_1,\ldots,N_r)}, L_p(\T^r))
    $$
    and the second case reduces to the first one.

    The general case $\Lambda=K\cap\mathbb{Z}^d$ follows from the
    GAP case and Theorem~\ref{thm_discrete_john}. Indeed, we have
    $\Lambda\subset P$, $|P|\le C(d)|\Lambda|$. We obtain the rank bound
    $|P|^{1-a} \le N^{1-a/2}$ for sufficiently large $N$, and the error bound
    $|P|^{-\alpha_p+\eta}\le N^{-\alpha_p+\eta}$.
\end{proof}

\subsection{$G$-circulants}

\begin{lemma}[$L_1$ circulant transfer]
\label{lem:circulant-transfer}
Let $F$ be an $N\times N$ matrix and suppose
$F=B+E$.
Let $D=\diag(d_k)$ be diagonal with
\[
 \sum_{k=0}^{N-1}|d_k|^2\le1.
\]
For every integer $1\le m<N$ there is a matrix $B'$ such that
\[
 \rank B'\le m+2\rank B
\]
and
\[
 \Delta_1(FDF-B')
 \le\frac{N}{\sqrt{m+1}}\,
 \Delta_1(E)^2.
\]
\end{lemma}

\begin{proof}
Let $S$ contain the $m$ largest values $|d_k|$ and write $D=D_1+D_0$, where
$D_1$ is supported on $S$.  Since $\sum_k|d_k|^2\le1$,
\[
 \rank D_1\le m,
    \qquad \Delta_\infty(D_0) \le\frac1{\sqrt{m+1}}.
\]
Set
\[
 B':=FD_1F+BD_0F+ED_0B.
\]
Then the estimate of $\rank B'$ is immediate and
\[
    FDF-B'=(FD_1F+FD_0F)-B'=(FD_1F+(B+E)D_0F)-B'=ED_0E.
\]
For every row $a$,
\[
 (ED_0E)_{a,*}
 =\sum_{k=0}^{N-1}E_{a,k}(D_0)_{k,k}E_{k,*},
\]
so
\[
 \begin{aligned}
 \|(ED_0E)_{a,*}\|_{L_1^N}
 &\le \sum_k |E_{a,k}|\frac{1}{\sqrt{m+1}}\|E_{k,*}\|_{L_1^N} \\
 &\le\frac{\Delta_1(E)}{\sqrt{m+1}}
       \sum_k|E_{a,k}|\\
 &\le\frac{N}{\sqrt{m+1}}\Delta_1(E)^2.
 \end{aligned}
\]
Taking the maximum over $a$ proves the lemma.
\end{proof}

Instead of circulants $\mathcal{C}_{G,f}(x,y)=f(x-y)$ it would be slightly
more convenient to work with matrices $A_{f}(x,y)=f(x+y)$ (we omit the $G$ from
the notation). They are equivalent to the
circulants up to columns permutation.

\begin{proof}[Proof of Proposition~\ref{prop:intro-circulant}]
Define the normalized Fourier transform
\[
 \widehat f(\xi)
 :=\frac1N\sum_{z\in G}f(z)e^{-2\pi i\xi\cdot z}.
\]
By Fourier inversion and symmetry of the pairing,
\[
 A_f=\mathcal{F}_GD_f\mathcal{F}_G,
 \qquad
 D_f:=\diag(\widehat f(\xi))_{\xi\in G},
\]
and Parseval's identity gives
\[
 \sum_{\xi\in G}|\widehat f(\xi)|^2\le1.
\]

Apply Theorem~\ref{thm:main-rank} with parameter $\eta/8$ and write
\[
 \mathcal{F}_G=B+E,
 \quad\rank B\le N^{1-a_1},
 \quad\Delta_1(E) \le N^{-1/2+\eta/8},
\]
for some $a_1>0$.
Choose
\[
 m:=\left\lfloor N^{1-\eta/2}\right\rfloor.
\]
Lemma~\ref{lem:circulant-transfer} gives a matrix $B_f$ with
\[
 \rank B_f
 \le N^{1-\eta/2}+2N^{1-a_1}
 \le N^{1-a}
\]
for some $a>0$ and all sufficiently large $N$.  Moreover,
\[
 \Delta_1(A_f-B_f)
 \le\frac{N}{\sqrt{m+1}} N^{-1+\eta/4}
 \le N^{-1/2+\eta/2}.
\]
\end{proof}

\begin{remark}[Sharpness of the power]\label{rem:circulant-sharpness}
Let $G=\Z_N$ with $N$ odd and put
$f(x):=e^{2\pi i x^2/N}$.
Then $|f|=1$, and for $x\ne x'$ one has
\begin{align*}
 \sum_{y\in\Z_N}f(x+y)\overline{f(x'+y)}
 &=e^{2\pi i(x^2-x'^2)/N}
   \sum_{y\in\Z_N}e^{4\pi i(x-x')y/N}=0,
\end{align*}
because multiplication by $2$ is invertible modulo $N$. Thus $N^{-1/2}A_f$ is unitary
and from~\eqref{unitary} we see
that the exponent $1/2$ in Proposition~\ref{prop:intro-circulant} is optimal.
\end{remark}

\subsection{Bilinear approximation}

\begin{proof}[Proof of Corollary~\ref{cor:intro-translation-kernels}]
Apply Proposition~\ref{prop:intro-circulant} with parameter $\eta/2$.
Put $M:=2N+1$, $x_j:=j/M$, $j=0,\ldots,M-1$, and consider the matrix
\[
 A_{ij}:=T(x_i-x_j).
\]
It is a cyclic circulant matrix of order $M$. After normalization by
$\|T\|_\infty$, Proposition~\ref{prop:intro-circulant} gives a matrix $B$ with
\[
 \rank B\le M^{1-a_0},
 \qquad
 \|A-B\|_{1,\infty}
 \lesssim \|T\|_\infty M^{1/2+\eta/2},
\]
where $a_0=a_0(\eta)>0$.  Hence, for every $p\ge1$,
\[
 \Delta_p(A-B)
 =M^{-1/p}\|A-B\|_{p,\infty}
 \le M^{-1/p}\|A-B\|_{1,\infty}
 \lesssim \|T\|_\infty N^{-\alpha_p+\eta/2}.
\]

Use the interpolation operator $P_N$ introduced above in both variables.
Since $T(x-y)$ belongs to $\mathcal T_N$ in each variable,
\[
 (P_N\otimes P_N)A=T(x-y).
\]
Write $B_{ij}=\sum_{\nu=1}^r a_i^{(\nu)}b_j^{(\nu)}$, $r=\rank B$.  Then
\[
 (P_N\otimes P_N)B
 =\sum_{\nu=1}^r (P_Na^{(\nu)})(x)(P_Nb^{(\nu)})(y),
\]
    so the rank does not increase. We use that $\|P_N\otimes P_N\|_{p\to
    p}\lesssim \log^2 N$ (analogously to the univariate case). Finally,
\[
 \begin{aligned}
 \|T(x-y)-(P_N\otimes P_N)B\|_{L_p(\T^2)}
 &\le \|P_N\|_{p\to p}^2\,\Delta_p(A-B)\\
 &\lesssim (\log N)^2\|T\|_\infty N^{-\alpha_p+\eta/2}\\
 &\lesssim \|T\|_\infty N^{-\alpha_p+\eta}.
 \end{aligned}
\]
Since $M\asymp N$, we have $r\le M^{1-a_0}\le N^{1-a_1}$ for some $a_1$.
\end{proof}

\subsection{The weighted Wiener classes.}

\begin{proof}[Proof of Corollary~\ref{cor:intro-trig-wiener}]
The lower bound uses only Theorem~\ref{thm_intro_lower}, so it is valid for any uniformly bounded orthonormal system.
Denote $e_k(x):=\exp(2\pi ikx)$.  For any $N$, we have
$$
    \mathcal{A}_\beta \supset N^{-\beta}\{e_0,\ldots,e_{N-1}\}.
$$
For $N=4n$ we obtain
$$
    d_n^{\mathbb{C}}(\mathcal{A}_\beta,L_p) \gtrsim N^{-\beta} N^{-\alpha_p} \asymp
    n^{-\beta-\alpha_p},
$$
and the choice $N\asymp n^{2/p}$ gives
$$
    d_n^{\mathbb{C}}(\mathcal{A}_\beta,L_p) \gtrsim N^{-\beta} \asymp n^{-2\beta/p} =
    n^{-\beta-2\beta\alpha_p}.
$$

For the upper estimate we apply Proposition~\ref{prop:consecutive-frequencies}
with $\eta=\alpha_p/2$. We obtain an approximating subspace $W$,
    $\dim W\le n/2$, and $N\asymp n^{1/(1-a)}$, $a=a(p)$,
    such that
    $$
    \dist(e_k,W)_{L_p}\lesssim N^{-\alpha_p/2},\quad |k|\le N.
    $$
    We take $V_n:=\Span\{e_k\colon |k| < n/4\}+W$.
    It follows that
    \begin{multline*}
    \dist(\mathcal{A}_\beta,V_n)_{L_p}
        \le \sup_{k\in\mathbb{Z}}\, (1+|k|)^{-\beta} \dist(e_k,V_n)_{L_p} \le \\
        \le \max\left\{\max_{n/4\le |k|\le N}(1+|k|)^{-\beta}\dist(e_k,W)_{L_p},\; N^{-\beta}\right\}
        \lesssim \max\{n^{-\beta}N^{-\alpha_p/2},N^{-\beta}\} \asymp
        n^{-\beta-b},
    \end{multline*}
    where $b=\min\{\alpha_p/2,a\beta\}/(1-a)$.
\end{proof}

\section*{AI assistance and contribution statement}

The research questions, the choice of problems, and the direction of the
investigation were formulated by Yuri Malykhin.  The proof ideas and draft
proofs assembled in the present version were generated in an iterative
dialogue with OpenAI's ChatGPT in response to those questions, with repeated
requests for verification, simplification, and recombination of the
arguments. Malykhin selected the results to be included and is responsible
for the independent mathematical verification and the final form of the
manuscript.


\begin{thebibliography}{99}

\bibitem{Alman21}
J.~Alman,
``Kronecker Products, Low-Depth Circuits, and Matrix Rigidity'',
\emph{Proc. 53rd Annual ACM SIGACT Symp. Th. Comput.} (STOC 2021), 772--785.

\bibitem{AlmanWilliams}
J.~Alman, R.~Williams,
``Probabilistic Rank and Matrix Rigidity'',
\emph{STOC 2017: Proc. 49th Ann. ACM SIGACT Symp. Th. Comput.},
2017, 641--652.

\bibitem{AstMal}
S.V.~Astashkin, Yu.V.~Malykhin,
``Rigidity of sets of independent functions in symmetric spaces'',
arXiv.2607.06530.

\bibitem{BhaskaraMahabadiVakilian2023}
A.~Bhaskara, S.~Mahabadi and A.~Vakilian,
``Tight bounds for volumetric spanners and applications'',
\emph{Advances in Neural Information Processing Systems} \textbf{36} (2023),
916--930.

\bibitem{DvirLiu2020}
Z.~Dvir and A.~Liu,
``Fourier and circulant matrices are not rigid'',
\emph{Theory Comput.} \textbf{16} (2020), 1--48.

\bibitem{Concrete}
R.~L. Graham, D.~E. Knuth, and O.~Patashnik,
\emph{Concrete Mathematics: A Foundation for Computer Science},
2nd ed., Addison--Wesley, 1994.

\bibitem{Hazan}
E.~Hazan, Z.~Karnin,
``Volumetric spanners: An Efficient Exploration Basis for Learning'',
\emph{J. Machine Learning Research} \textbf{17} (2016), 1--34.

\bibitem{KMR}
B.~S. Kashin, Yu.~V. Malykhin, K.~S. Ryutin,
``Kolmogorov Width and Approximate Rank'',
\emph{Proceedings of the Steklov Institute of Mathematics} \textbf{303} (2018),
140--153.

\bibitem{Kivva}
B.~Kivva,
``Improved Upper Bounds for the Rigidity of Kronecker Products'',
\emph{Proc. 46th Int. Symp. Math. Found. Comp. Sci.} (MFCS 2021), Vol. 202,
68:1--68:18.

\bibitem{Malykhin2022}
Y.~Malykhin,
``Matrix and tensor rigidity and $L_p$-approximation'',
\emph{J. Complexity} \textbf{72} (2022), 101651, 13 pp.

\bibitem{Malykhin2024}
Y.~Malykhin,
``Widths and rigidity'',
\emph{Mat. Sb.} \textbf{215} (2024), no.~4, 117--148 (in Russian).

\bibitem{MalRyutin}
Yu.V.~Malykhin, K.S.~Ryutin,
``Widths and rigidity of unconditional sets and random vectors'',
\emph{Izvestiya: Mathematics} \textbf{89}:2 (2025), 261--273.

\bibitem{MalRyutin2}
Yu.V.~Malykhin, K.S.~Ryutin,
``Kolmogorov Widths of Balls in Mixed Norms: The Case of Rigidity'',
\emph{Proc. Steklov Inst.} \textbf{331} (2025), 110--117.

\bibitem{MoellerStasyukUllrich2026}
M.~Moeller, S.~Stasyuk and T.~Ullrich,
``Best $m$-term trigonometric approximation in weighted Wiener spaces and applications'',
\emph{Adv. Oper. Theory} \textbf{11} (2026), article 18.

\bibitem{NguyenNguyenSickel2022}
V.~D.~Nguyen, V.~K.~Nguyen and W.~Sickel,
``$s$-Numbers of embeddings of weighted Wiener algebras'',
\emph{J. Approx. Theory} \textbf{279} (2022), 105745.

\bibitem{RudinGroups}
W.~Rudin,
\emph{Fourier Analysis on Groups},
Interscience Publishers, New York--London, 1962.

\bibitem{RudinRealComplex}
W.~Rudin,
\emph{Real and complex analysis},
3rd ed., McGraw-Hill, 1987.

\bibitem{Sinai}
A.~Sinai,
``Nonrigidity of the trigonometric system in $L_p$ for $1\le p<2$'',
arXiv.2608.12170.

\bibitem{TaoVu2008}
T.~Tao and V.~Vu,
``John-type theorems for generalized arithmetic progressions and iterated
sumsets'',
\emph{Adv. Math.} \textbf{219} (2008), no.~2, 428--449.

\bibitem{Temlyakov2003}
V.~N. Temlyakov,
``Nonlinear Methods of Approximation'',
\emph{Found. Comput. Math.} \textbf{3} (2003), 33--107.

\bibitem{Valiant1977}
L.~G. Valiant,
``Graph-theoretic arguments in low-level complexity'',
\emph{Mathematical Foundations of Computer Science (MFCS 1977)}, 162--176.

\end{thebibliography}
\end{document}